\documentclass[a4paper, 11 pt]{article}
\usepackage[margin=1 in]{geometry}       
\usepackage{graphics}             
\usepackage{multicol}
\usepackage{epsfig,color}
\usepackage{color}
\usepackage{slashbox}
\usepackage{amsmath}
\usepackage{graphicx}
\usepackage{caption}
\usepackage{subcaption}
\usepackage{float}
\usepackage{amsthm}
\usepackage{mathrsfs}
\usepackage[nottoc,notlot,notlof]{tocbibind}
\usepackage{tocloft}
\usepackage{mathtools}
\usepackage{tikz}
\usetikzlibrary{arrows}
\usepackage{caption}
\newcommand{\midarrowa}{\tikz \draw[-triangle 45] (0,0) -- +(.1,0);}  
\tikzset{every picture/.style={line width=0.9pt}}

\providecommand{\keywords}[1]
{
  \small	
  \textbf{\textit{Keywords---}} #1
}
 	\author{Jasine Babu, Deepu Benson, Deepak Rajendraprasad and Sai Nishant Vaka \\
        \small Indian Institute of Technology Palakkad \\}
\date{}

	\usepackage{enumerate}
	\usepackage{xcolor}
	
	\usepackage{amsmath,amsthm, amssymb}

	\newtheorem{theorem}{Theorem}
	\newtheorem{corollary}[theorem]{Corollary}
	\newtheorem{lemma}[theorem]{Lemma}

	\theoremstyle{definition}

	\theoremstyle{remark}
	\newtheorem{remark}{Remark}
	
	\def\--{\mbox{--}}

	\newcommand{\ceil}[1]{\left\lceil #1 \right\rceil}
	\newcommand{\floor}[1]{\left\lfloor #1 \right\rfloor}
	\newcommand{\orient}[1]{\vec{#1}}

\usepackage{fullpage}
\usepackage{times}
\usepackage{fancyhdr,amssymb}
\usepackage[ruled,vlined]{algorithm2e}
\DeclareRobustCommand*{\ora}{\vec}

\tocloftpagestyle{empty}

\title{An Improvement to Chv{\'a}tal and Thomassen's Upper Bound for Oriented Diameter}
               
\cleardoublepage 
\begin{document}  

\maketitle
\begin{abstract}
An orientation of an undirected graph $G$ is an assignment of exactly one direction to each edge of $G$. The oriented diameter of a graph $G$ is the smallest diameter among all the orientations of $G$. The maximum oriented diameter of a family of graphs $\mathscr{F}$ is the maximum oriented diameter among all the graphs in $\mathscr{F}$. Chv{\'a}tal and Thomassen [JCTB, 1978] gave a lower bound of $\frac{1}{2}d^2+d$ and an upper bound of $2d^2+2d$ for the maximum oriented diameter of the family of $2$-edge connected graphs of diameter $d$. We improve this upper bound  to $ 1.373 d^2 + 6.971d-1 $, which outperforms the former upper bound for all values of $d$ greater than or equal to $8$. For the family of $2$-edge connected graphs of diameter $3$, Kwok, Liu and West [JCTB, 2010] obtained improved lower and upper bounds of $9$ and $11$ respectively. For the family of $2$-edge connected graphs of diameter $4$, the bounds provided by Chv{\'a}tal and Thomassen are $12$ and $40$ and no better bounds were known. By extending the method we used for diameter $d$ graphs, along with  an asymmetric extension of a technique used by Chv{\'a}tal and Thomassen, we have improved this upper bound to $21$.
\end{abstract} \hspace{10pt}


\section{Introduction}

An \emph{orientation} of an undirected graph $G$ is an assignment of exactly one direction to each of the edges of $G$. A given undirected graph can be oriented in many different ways ($2^m$, to be precise, where m is the number of edges). The studies on graph orientations often concern with finding orientations which achieve a predefined objective. Some of the objectives while orienting graphs include minimization of certain distances, ensuring acyclicity, minimizing the maximum in-degree, maximizing connectivity, etc. One of the earliest studies regarding graph orientations were carried out by H.E. Robbins in 1939. He was trying to answer a question posed by Stanislaw Ulam. ``\textit{When may the arcs of a graph be so oriented that one may pass from any vertex to any other, traversing arcs in the positive sense only?}". This led to a seminal work \cite{robbins1939theorem}  of Robbins in which he proved the following theorem, ``\textit{A graph is orientable if and only if it remains connected after the removal of any arc}"'. 

A directed graph $\ora{G}$ is called \emph{strongly connected} if it is possible to reach any vertex starting from any other vertex using a directed path. An undirected graph $G$ is called \emph{strongly orientable} if it has a strongly connected orientation. A \emph{bridge} in a connected graph is an edge whose removal will disconnect the graph. A \emph{$2$-edge connected} graph is a connected graph which does not contain any bridges.  The theorem of Robbins stated earlier says that it is possible for a graph $G$ to be strongly oriented if and only if $G$ is $2$-edge connected. Though Robbins stated the necessary and sufficient conditions for a graph to have a strong orientation, no comparison between the diameter of a graph and the diameter of an orientation of this graph was given in this study. This was taken up by Chv{\'a}tal and Thomassen in 1978 \cite{chvatal1978distances}.  

In order to discuss these quantitative results, we introduce some notation. Let $G$ be an undirected graph. The \emph{distance} between two vertices $u$ and $v$ of $G$, $d_G(u,v)$ is the number of edges in a shortest path between $u$ and $v$. For any two subsets $A$, $B$ of $V(G)$, let $d_G(A,B) = \min\{d_G(u,v) : u \in A, v \in B \}$. The \emph{eccentricity} of a vertex $v$ of $G$ is the maximum distance between $v$ and any other vertex $u$ of $G$. The \emph{diameter} of $G$ is the maximum of the eccentricities  of its vertices. The \emph{radius} of $G$ is the minimum  of the eccentricities of its vertices. Let $\ora{G}$ be a directed graph and $u,v \in V(\ora{G})$. Then the \emph{distance} from a vertex $u$ to $v$, $d_{\ora{G}}(u,v)$, is defined as the length of a shortest directed path from $u$ to $v$. For any two subsets $A$, $B$ of $V(\ora{G})$, let $d_{\ora{G}}(A,B)=\min\{d_{\ora{G}}(u,v) : u \in A, v \in B \}$. The \emph{out-eccentricity} of a vertex $v$ of $\ora{G}$ is the greatest distance from $v$ to a vertex $u \in V(\ora{G})$. The \emph{in-eccentricity} of a vertex $v$ of $\ora{G}$ is the greatest distance from a vertex $u \in V(\ora{G})$ to $v$. The \emph{eccentricity} of a vertex $v$ of $\ora{G}$ is the maximum of its out-eccentricity and in-eccentricity. The \emph{diameter} of  $\ora{G}$, denoted by $d(\ora{G})$, is the maximum of the eccentricities  of its vertices. The \emph{radius} of $\ora{G}$ is the minimum  of the eccentricities of its vertices. The \emph{oriented diameter} of an undirected graph $G$, denoted by $\ora{d}(G)$, is the smallest diameter among all strong orientations of $G$. That is, $ \ora{d}(G)  := \text{min}\{d(\ora{G}): \ora{G} \text{ is an orientation of } $G$\} $. The \emph{oriented radius} of an undirected graph $G$ is the smallest radius among all strong orientations of $G$. The maximum oriented diameter of the family $\mathscr{F}$ of graphs is the maximum oriented diameter among all the graphs in $\mathscr{F}$. Let $f(d)$ denote the maximum oriented diameter of the family of $2$-edge connected diameter $d$ graphs. That is, $ f(d)  := \text{max}\{\ora{d}(G): G \in \mathscr{F}\} $, where $\mathscr{F}$ is the family of $2$-edge connected graphs with diameter $d$. 

The following theorem by Chv{\'a}tal and Thomassen \cite{chvatal1978distances} gives an upper bound for the oriented radius of a graph.
\begin{theorem}\label{tt1}\cite{chvatal1978distances}
Every 2-edge connected graph of radius $r$ admits a strong orientation of radius at most $r^2+r$.
\end{theorem}    
The above bound was also shown to be tight. In the same paper, they also proved that the problem of deciding whether an undirected graph admits an orientation of diameter $2$ is NP-hard. Motivated by the work of  Chv{\'a}tal and Thomassen \cite{chvatal1978distances}, Chung, Garey and Tarjan \cite{chung1985strongly} proposed a linear-time algorithm to check whether a mixed multigraph has a strong orientation or not. They have also proposed a polynomial time algorithm which provides a strong orientation (if it exists) for a mixed multigraph with oriented radius at most $4r^2 + 4r$. Studies have also been carried out regarding the oriented diameter of specific subclasses of graphs like AT-free graphs, interval graphs, chordal graphs and planar graphs  \cite{fomin2004free,fomin2004complexity,eggemann2009minimizing}. Bounds on oriented diameter in terms of other graph parameters like minimum degree and maximum degree are also available in literature   
\cite{bau2015diameter, surmacs2017improved,  erdHos1989radius, dankelmann2016diameter}.

The following bounds for $f(d)$ were given by Chv{\'a}tal and Thomassen \cite{chvatal1978distances}.

\begin{theorem}\label{CTboundD}\cite{chvatal1978distances}
$ \frac{1}{2} d^{2} + d \leq f(d) \leq 2d^{2} + 2d $. 
\end{theorem}
Chv{\'a}tal and Thomassen \cite{chvatal1978distances} has also proved that $f(2) = 6$. By Theorem~\ref{CTboundD}, $8 \leq f(3) \leq 24$. In 2010, Kwok, Liu and West \cite{kwok2010oriented} improved these bounds to $9 \leq f(3) \leq 11$. To prove the upper bound of $11$, Kwok, Liu and West partitioned the vertices of $G$ into a number of sets based on the distances from the endpoints of an edge which is not part of any $3$-cycle. Our study on the oriented diameter of $2$-edge connected graphs with diameter $d$ uses this idea of partitioning the vertex set into a number of sets based on their distances from a specific edge.

\subsection*{Our Results}

In this paper we establish two improved upper bounds. Firstly in Section 2, we show that $f(d) \leq 1.373d^2 + 6.971d - 1$ (Theorem~\ref{betterthanct}). This is the first general improvement to Chv{\'a}tal and Thomassen's upper bound $f(d) \leq 2d^2 + 2d$ from 1978. For all $d \geq 8$, our upper bound outperforms that of Chv{\'a}tal and Thomassen. Their lower bound $f(d) \geq \frac{1}{2} d^{2} + d$ still remains unimproved. We do not believe that our upper bound is tight.  Secondly in Section 3, for the case of $d = 4$, we further sharpen our analysis and show that $f(4) \leq 21$ (Theorem~\ref{betterfor4}). This is a considerable improvement from $40$, which follows from Chv{\'a}tal and Thomassen's general upper bound. Here too, our upper bound is not yet close to the lower bound of $12$ given by Chv{\'a}tal and Thomassen and we believe that there is room for improvement in the upper bound. 
  
\section{Oriented Diameter of Diameter $\boldsymbol{d}$ Graphs}
A subset $D$ of the vertex set of $G$ is called a \emph{$k$-step dominating set} of $G$ if every vertex not in $D$ is at a distance of at most $k$ from at least one vertex of $D$. An oriented subgraph $\ora{H}$ of $G$ is called a $k$-step dominating oriented subgraph if $V(\ora{H})$ is a $k$-step dominating set of $V(G)$. To obtain upper bounds for the oriented diameter of a graph $G$ with $n$ vertices and minimum degree $\delta \geq 2$, Bau and Dankelmann \cite{bau2015diameter} and Surmacs \cite{surmacs2017improved} first constructed a $2$-step dominating oriented subgraph $\ora{H}$ of $G$. They used this together with the idea in the proof of Theorem~\ref{tt1} on $\ora{H}$ to obtain the upper bounds of $\frac{11n}{\delta+1}+9$ and $\frac{7n}{\delta +1}$, respectively, for the oriented diameter of graphs with minimum degree $\delta \geq 2$. We are using the algorithm {\sc OrientedCore} described below to produce a $2$-edge connected oriented subgraph $\ora{H}$ of $G$ with some distance guarantees between the vertices in $\ora{H}$ (Lemma~\ref{lemmaDistancesInCore}) and some domination properties (Lemma~\ref{lemmaDominationByCore}).

\subsection{Algorithm {\sc OrientedCore}}

\newcommand{\emsec}[1]{
			\medskip
			\noindent
			\emph{#1}
			}

\emsec{Input:} A $2$-edge connected graph $G$ and a specified edge $pq$ in $G$.

\emsec{Output:} A $2$-edge connected oriented subgraph $\orient{H}$ of $G$.

\emsec{Terminology:} Let $d$ be the diameter of $G$, let $k$ be the length of a smallest cycle
	containing $pq$ in $G$ and let $h = \floor{k/2}$. Notice that $k \leq 2d+1$ and $h \leq d$. Define $S_{i,j} = \{v \in
	V(G) : d_G(v,p) = i, d_G(v,q) = j\}$. Since $S_{i,j}$ is non-empty only if $0
	\leq i, j \leq d$ and $|i-j| \leq 1$, we implicitly assume these
	restrictions on the subscripts of $S_{i,j}$ wherever we use it. For a
	vertex $v \in S_{i,j}$, its \emph{level} $L(v)$ is $(j-i)$ and its
	\emph{width} $W(v)$ is $\max(i,j)$. We will always refer to an edge $\{u,v\}$ between two different $S_{i,j}$'s as $uv$ when either $L(u) > L(v)$ or $L(u) =
	L(v)$ and $W(u) < W(v)$ (downward or rightward in Fig.~\ref{fig1}).  Moreover
	the edge $uv$ is called {\em vertical} in the first case and {\em
	horizontal} in the second. 

\emsec{Observations based on the first edge of shortest paths from a
vertex $v$ to $p$ or $v$ to $q$:} Every vertex $v \in S_{i,i+1}$, $1 \leq i \leq d-1$, is incident to a
	horizontal edge $uv$ with $u \in S_{i-1, i}$.  Every vertex $v \in
	S_{i+1,i}$, $1 \leq i \leq d-1$, is incident to a horizontal edge $uv$ with
	$u \in S_{i,i-1}$.  Every vertex $v \in S_{i,i}$, $1 \leq i \leq d$, is
	incident either to a horizontal edge $uv$ with $u \in S_{i-1,i-1}$ or two
	vertical edges $uv$ and $vx$ with $u \in S_{i-1,i}$ and $x \in S_{i, i-1}$.
	Consequently for any $v$ in Level $1$, all the shortest $p\--v$ path consists
	of Level $1$ horizontal edges only and for any vertex in $v$ in Level $1$, all
	the shortest $v\--q$ path consists of Level $-1$ horizontal edges alone. For
	any vertex $v$ in Level $0$, all the shortest $p\--v$ path consists of
	horizontal edges in levels $1$ and $0$ and exactly one vertical edge; while all
	the shortest $v\--q$ path consists of horizontal edges in levels $0$ and $-1$
	and exactly one vertical edge.

\emsec{Stage $1$.} Initialise $\orient{H}$ to be empty.  For each vertical edge $uv$ with
	$L(u) = 1$ and $L(v) \in \{0,-1\}$, and for each shortest $p\--u$ path
	$P_u$ and shortest $v\--q$ path $P_v$, do the following: Let $P$ be the
	$p\--q$ path formed by joining $P_u$, the edge $uv$ and $P_v$. Orient the
	path $P$ as a directed path $\orient{P}$ from $p$ to $q$ and add it to
	$\orient{H}$. Notice that even though two such paths can share edges, there
	is no conflict in the above orientation since, in Stage~$1$, every vertical
	edge is oriented downward, every horizontal edge in Level~$1$ is oriented
	rightward and every horizontal edge in levels $0$ and $-1$ is oriented
	leftward. 

\emsec{Stage~$2$.} For each vertical edge $uv$ with $L(u) = 0$ and $L(v) = -1$ not already oriented in Stage~$1$, and for each
	shortest $p\--u$ path $P_u$ and shortest $v\--q$ path $P_v$ do the
	following: Let $x$ be the last vertex in $P_u$ (nearest to $u$) that is
	already in $V(\orient{H})$ and let $P_u'$ be the subpath of $P_u$ from $x$
	to $u$.  Similarly let $y$ be the first vertex in $P_v$ (nearest to $v$)
	that is already in $V(\orient{H})$ and let $P_v'$ be the subpath of $P_v$
	from $v$ to $y$. Let $P$ be the $x\--y$ path formed by joining $P_u'$, the
	edge $uv$ and $P_v'$.  Orient the path $P$ as a directed path $\orient{P}$
	from $x$ to $y$ and add it to $\orient{H}$.  Notice that $P$ does not share
	any edge with a path added to $\orient{H}$ in Stage~$1$, but it can share
	edges with paths added in earlier steps of Stage~$2$. However there is no
	conflict in the orientation since, in Stage~$2$, every vertical edge is
	oriented downward, every horizontal edge in Level~$0$ is oriented
	rightward, every horizontal edge in Level~$-1$ is oriented leftward, and no
	horizontal edges in Level~$1$ is added.

\emsec{Stage~$3$.} Finally orient the edge $pq$ from $q$ to $p$ and add it to $\orient{H}$.
	This completes the construction of $\orient{H}$, the output of the 
	algorithm.

\subsubsection*{Distances in $\boldsymbol{\orient{H}}$}

	First we analyse the (directed) distance from $p$ and to $q$ of vertices
	added to $\orient{H}$ in Stage~$1$.  The following bounds on distances in
	$\orient{H}$ follow from the construction of each path $P$ in Stage~1. Let
	$w$ be any vertex that is added to $\orient{H}$ in Stage~$1$. Then

	\begin{equation}
	\label{eqnStage1Fromp}
	d_{\orient{H}}(p, w) \leq 
		\begin{cases}
			\begin{array}{ll}
				i,		& w \in S_{i,i+1}, \\
				h,		& w \in S_{h,h}, \\
				2d - i,	& w \in S_{i,i},~ i > h, \text{ and} \\
				2d - i,	& w \in S_{i+1,i}. \\
			\end{array}
		\end{cases}
	\end{equation}

	\begin{equation}
	\label{eqnStage1Toq}
	d_{\orient{H}}(w,q) \leq 
		\begin{cases}
			\begin{array}{ll}
				2d - i,	& w \in S_{i,i+1}, \\
				h,		& w \in S_{h,h}, \\
				2d - i,	& w \in S_{i,i},~ i > h, \text{ and} \\
				i,		& w \in S_{i+1,i}. \\
			\end{array}
		\end{cases}
	\end{equation}

	It is easy to verify the above equations using the facts that $w$ is part
	of a directed $p\--q$ path of length at most $2d$ (at most $2h$ if $w \in
	S_{h,h}$) in $\orient{H}$.

	No new vertices from Level~$1$ or $S_{h,h}$ are added to $\orient{H}$ in
	Stage~$2$. Still the distance bounds for vertices added in Stage~$2$ are
	slightly more complicated since a path $P$ added in this stage will start
	from a vertex $x$ in Level~$0$ and end in a vertex $y$ in Level $-1$, which
	are added to $\orient{H}$ in Stage~$1$. But we can complete the analysis
	since we already know that $d_{\orient{H}}(p,x) \leq 2d - h -1$ and
	$d_{\orient{H}}(y, q) \leq i$ where $i$ is such that $y \in S_{i+1,i}$ from the analysis of Stage~1.  Let $w$ be any
	vertex that is added to $\orient{H}$ in Stage~$2$. Then

	\begin{equation}
	\label{eqnStage2Fromp}
	d_{\orient{H}}(p, w) \leq 
		\begin{cases}
				(2d-h- 1) + (i-h-1) 	\\
					\quad= 2d - 2h - 2 + i,	\quad 	w \in S_{i,i},~ i > h, 
					\text{ and} \\
				(2d-h- 1) + (d-h-1) + (d-i)	\\
					\quad= 4d - 2h - 2 - i,	\quad	w \in S_{i+1,i}. \\
		\end{cases}
	\end{equation}

	The distance from $w$ to $q$ in $\orient{H}$ is not affected even though
	we trim the path $P_v$ at $y$ since $y$ already has a directed shortest
	path to $q$ from Stage~$1$. Hence

	\begin{equation}
	\label{eqnStage2Toq}
	d_{\orient{H}}(w,q) \leq 
		\begin{cases}
			\begin{array}{ll}
				2d - i,	& w \in S_{i,i},~ i > h, \text{ and} \\
				i,		& w \in S_{i+1,i}. \\
			\end{array}
		\end{cases}
	\end{equation}

	The first part of the next lemma follows from taking the worst case among
	(\ref{eqnStage1Fromp}) and (\ref{eqnStage2Fromp}). Notice that $\forall i >
	h,~ (2h+2-i \leq i)$ and $(4d - 2h -2 \geq 2d)$ when $h < d$. New
	vertices are added to $\orient{H}$ in Stage~2 only if $h < d$. The second
	part follows from (\ref{eqnStage1Toq}) and (\ref{eqnStage2Toq}). The
	subsequent two claims are easy observations.

\begin{lemma}
\label{lemmaDistancesInCore}
	Let $G$ be a $2$-edge connected graph, $pq$ be any edge of $G$ and let
	$\orient{H}$ be the oriented subgraph of $G$ returned by the algorithm {\sc
	OrientedCore}. Then for every vertex $w \in V(\orient{H})$ we have 

	\begin{equation}
	\label{eqnFromp}
	d_{\orient{H}}(p, w) \leq 
		\begin{cases}
			\begin{array}{ll}
				i,				& w \in S_{i,i+1}, \\
				h,				& w \in S_{h,h}, \\
				2d - 2h - 2 + i,	& w \in S_{i,i},~ i > h, \text{ and} \\
				4d - 2h - 2 - i,	& w \in S_{i+1,i}. \\
			\end{array}
		\end{cases}
	\end{equation}

	\begin{equation}
	\label{eqnToq}
	d_{\orient{H}}(w,q) \leq 
		\begin{cases}
			\begin{array}{ll}
				2d - i,	& w \in S_{i,i+1}, \\
				h,		& w \in S_{h,h}, \\
				2d - i,	& w \in S_{i,i},~ i > h, \text{ and} \\
				i,		& w \in S_{i+1,i}. \\
			\end{array}
		\end{cases}
	\end{equation}

	Moreover, $d_{\orient{H}}(q,p) = 1$ and  $d_{\orient{H}}(p,q) \leq k-1$. 
\end{lemma}
We can see that if $S_{h,h}$ is non-empty, then all the vertices in $S_{h,h}$ are captured into $\ora{H}$.  

Notice that when $k \geq 4$, $S_{1,2}$ and $S_{2,1}$ are non empty.
Thus the bound on the diameter of
	$\orient{H}$ follows by the triangle inequality $d_{\orient{H}}(x,y) \leq
	d_{\orient{H}}(x, q) + d_{\orient{H}}(q,p) + d_{\orient{H}}(p,y)$ and the
	fact that $\forall k \geq 4$ the worst bounds for $d_{\orient{H}}(x, q)$ and
	$d_{\orient{H}}(p, y)$ from Lemma~\ref{lemmaDistancesInCore} are when $x \in S_{1,2}$ and $y \in S_{2,1}$. 
Hence the following corollary. 
\begin{corollary}\label{ddcorollary}
Let $G$ be a $2$-edge connected graph, $pq$ be any edge of $G$ and let
	$\orient{H}$ be the oriented subgraph of $G$ returned by the algorithm {\sc
	OrientedCore}. If the length of the smallest cycle containing $pq$ is greater than or equal to $4$,
then the diameter of $\orient{H}$ is at most $6d-2h-3$.
\end{corollary}

\subsubsection*{Domination by $\boldsymbol{\orient{H}}$} 
	
	Let us call the vertices in $V(\orient{H})$ as \emph{captured} and those in
	$V(G) \setminus V(\orient{H})$ as \emph{uncaptured}.  For each $i \in \{1,
	0, -1\}$ let $L_i^c$ and $L_i^u$ denote the captured and uncaptured
	vertices in level i, respectively.  Since $L_i^c$ contains every level $i$
	vertex incident with a vertical edge, $L_i^c$ separates $L_i^u$ from rest
	of $G$. Let $d_i$ denote the maximum distance between a vertex in $L_i^u$
	and the set $L_i^c$. If $u_i \in L_i^u$ and $u_j \in L_j^u$ such that $d_G(u_i, L_i^c) = d_i$ and $d_G(u_j, L_j^c) = d_j$ for distinct $i, j \in \{1,0,-1\}$, the distance $d_G(u_i, u_j)$ is bounded
	above by $d$, the diameter of $G$, and bounded below by $d_i+ 1+d_j$. Hence
	$d_i + d_j \leq d-1$ for every distinct $i, j \in \{1,0,-1\}$.

	For any vertex $u \in L_0^u$, the last Level $0$ vertex in a shortest
	(undirected) $u\--q$ path is in $L_0^c$. Hence if Level $0$ is non-empty then $ d_0 \leq (d - h)$.  In
	order to bound $d_1$ and $d_{-1}$, we take a close look at a shortest cycle
	$C$ containing the edge $pq$. Let $C = (v_1, \ldots, v_{k}, v_1)$ with $v_1
	= q$ and $v_{k} = p$.  Each $v_i$ is in $S_{i,i-1}$ when $2i < k+1$,
	$S_{i-1,i-1}$ if $2i = k+1$ and $S_{k-i, k-i+1}$ when $2i > k+1$. Let $t =
	\ceil{k/4}$. The Level $-1$ vertex $v_t$ is special since it is at a distance
	$t$ from Level $1$ and thus $L_1^c$. If $u_1$ is a vertex in $L_1^u$ such that $d_G(u_1, L_1^c) = d_1$, the
	distance $d_G(u_1, v_t)$ is bounded above by $d$ and below by $d_1+t$.
	Hence $d_1 \leq d - t$.  Similarly we can see that $d_{-1} \leq (d-t)$.


	Putting all these distance bounds on domination together, we get the next
	lemma.

\begin{lemma}
\label{lemmaDominationByCore}

	Let $G$ be a $2$-edge connected graph, $pq$ be any edge of $G$ and let
	$\orient{H}$ be the oriented subgraph of $G$ returned by the algorithm {\sc
	OrientedCore}. For each $i \in \{1,0,-1\}$, let $d_i$ denote the maximum
	distance of a level $i$ vertex not in $V(\orient{H})$ to the set of level
	$i$ vertices in $V(\orient{H})$. Then $d_0 \leq d - \floor{k/2}$, $d_1,
	d_{-1} \leq d - \ceil{k/4}$ and for any distinct $i, j \in \{1, 0, -1\}$,
	$d_i + d_j \leq d-1$.

\end{lemma}

\subsection{The Upper Bound}

Consider a $2$-edge connected graph $G$ with diameter $d$. Let $\eta(G)$ denote the smallest integer such that every edge of a graph $G$ belongs to a cycle
of length at most $\eta(G)$. Sun, Li, Li and Huang \cite{sun2014oriented} proved the following theorem. 

\begin{theorem}\cite{sun2014oriented}
$\ora{d}(G) \leq 2r(\eta-1)$ where $r$ is the radius of $G$ and $\eta = \eta(G)$. 

\end{theorem}  

We know that $r \leq d$ and hence we have $\ora{d}(G) \leq 2d(\eta-1)$ as our first bound. Let $pq$ be an edge in $G$ such that the length of a smallest cycle containing $pq$ is $\eta$. If $\eta \leq 3$, then $\ora{d}(G) \leq 4d$ which is smaller than the bound claimed in Theorem~\ref{betterthanct}. So we assume $\eta \geq 4$. By Corollary~\ref{ddcorollary}, $G$ has an oriented subgraph $\ora{H}$ with diameter at most $6d-2\floor{\frac{\eta}{2}}-3$. Moreover by Lemma~\ref{lemmaDominationByCore}, $\ora{H}$ is a $(d - \ceil{\frac{\eta}{4}})$-step dominating subgraph of $G$. Let $G_0$ be a graph obtained by contracting the vertices in $V(\ora{H})$ into a single vertex $v_H$. We can see that $G_0$ has radius at most $(d - \ceil{\frac{\eta}{4}})$. Thus by Theorem~\ref{tt1}, $G_0$ has a strong orientation $\ora{G_0}$ with radius at most $(d-\ceil{\frac{\eta}{4}})^2 + (d-\ceil{\frac{\eta}{4}})$. Since $d \leq 2r$, we have $d(\ora{G_0}) \leq 2(d-\ceil{\frac{\eta}{4}})^2 + 2(d-\ceil{\frac{\eta}{4}})$. Notice that $\ora{G_0}$ and $\ora{H}$ do not have any common edges. Hence $G$ has an orientation with diameter at most $2(d-\ceil{\frac{\eta}{4}})^2 + 2(d-\ceil{\frac{\eta}{4}}) + (6d- 2\floor{\frac{\eta}{2}}-3)$ by combining the orientations in $\ora{H}$ and $\ora{G_0}$. Let $\eta = 4 \alpha d$. Hence we get $\ora{d}(G) \leq \min\{8\alpha d^2 - 2d, 2(1-\alpha)^2d^2 + 8d-6\alpha d-1\}$. We can see that the dominant term in the first bound is $8 \alpha d^2$ while the dominant term in the second bound is at most $2(1-\alpha)^2 d^2$. Notice that $0 < \frac{3}{4d} \leq \alpha \leq \frac{2d+1}{4d} < 1$. Thus by optimizing for $\alpha$ in the range $(0,1)$, we obtain the following theorem.

\begin{theorem}\label{betterthanct}
$ f(d) \leq 1.373 d^2 + 6.971 d - 1  $.
\end{theorem}

For any $d \geq 8$, the above upper bound is an improvement over the upper bound of $2 d ^2 + 2d$ provided by Chv{\'a}tal and Thomassen.

\section{Oriented Diameter of Diameter $\boldsymbol{4}$ Graphs}

Throughout this section, we consider $G$ to be an arbitrary $2$-edge connected diameter $4$ graph. We will show that the oriented diameter of $G$ is at most $21$ and hence $f(4) \leq 21$. The following lemma by Chv{\'a}tal and Thomassen \cite{chvatal1978distances} is used when $\eta(G) \leq 4$.   

\begin{lemma}\label{LemmaforK}\cite{chvatal1978distances}
Let $\Gamma$ be a 2-edge connected graph. If every edge of $\Gamma$ lies in a cycle of length at most $k$, then it has an orientation $\ora{\Gamma}$ such that
     $$ d_{\ora{\Gamma}}(u,v) \leq ((k-2) 2^{\left \lfloor (k-1)/2 \right \rfloor} + 1)  d_{\Gamma}(u,v) \quad \forall u,v \in V(\ora{\Gamma}) $$
\end{lemma}

Hence if all edges of the graph $G$ lie in a $3$-cycle or a $4$-cycle, the oriented diameter of $G$ will be at most $20$. Hence we can assume the existence of an edge $pq$ which is not part of any $3$-cycle or $4$-cycle as long as we are trying to prove an upper bound of $20$ or more for $f(4)$. We apply algorithm {\sc OrientdCore} on $G$ with the edge $pq$ to obtain an oriented subgraph $\ora{H_1}$ of $G$. Fig.~\ref{fig1} shows a coarse representation of $\ora{H_1}$.     

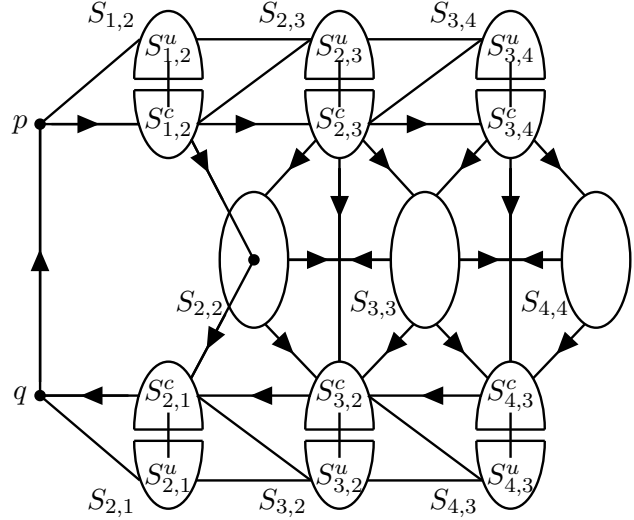
\begin{figure}[h]
    
    \centering
   
\begin{tikzpicture}[scale=0.75]
\begin{scope}[every node/.style={sloped,allow upside down}]
\draw[fill=black] (0.75,0.6) circle (0.08cm);
 
\draw[fill=black] (0.75,5.4) circle (0.08cm);

\draw (0.75,0.6)--(0.75,5.4);

\draw (0.75,0.84)--node {\midarrowa}(0.75,5.4);


\node at (3,6.7) {$S_{1,2}^u$};
\node at (6,6.7) {$S_{2,3}^u$};
\node at (9,6.7) {$S_{3,4}^u$};
\node at (3,-0.9) {$S_{2,1}^u$};
\node at (6,-0.9) {$S_{3,2}^u$};
\node at (9,-0.9) {$S_{4,3}^u$};

\draw (3,6.5)--(3,5.7);
\draw (6,6.5)--(6,5.7);
\draw (9,6.5)--(9,5.7);

\draw (3,-0.5)--(3,0.3);
\draw (6,-0.5)--(6,0.3);
\draw (9,-0.5)--(9,0.3);

\draw (2.4,6.2)--(3.6,6.2);
\draw (5.4,6.2)--(6.6,6.2);
\draw (8.4,6.2)--(9.6,6.2);

\draw (2.4,-0.2)--(3.6,-0.2);
\draw (5.4,-0.2)--(6.6,-0.2);
\draw (8.4,-0.2)--(9.6,-0.2);

\draw (0.75,5.4)--(2.5,6.9);
\draw (3.5,6.9)--(5.5,6.9);
\draw (6.5,6.9)--(8.5,6.9);

\draw (0.75,0.6)--(2.5,-0.9);
\draw (3.5,-0.9)--(5.5,-0.9);
\draw (6.5,-0.9)--(8.5,-0.9);

\draw (2.4,6) arc(180:360:0.6cm and 1.2cm);
\draw (3.6,6.2) arc(0:180:0.6cm and 1.2cm);

\draw (2.4,-0.2) arc(180:360:0.6cm and 1.2cm);
\draw (3.6,0) arc(0:180:0.6cm and 1.2cm);

\draw (5.4,6) arc(180:360:0.6cm and 1.2cm);
\draw (6.6,6.2) arc(0:180:0.6cm and 1.2cm);

\draw (5.4,-0.2) arc(180:360:0.6cm and 1.2cm);
\draw (6.6,0) arc(0:180:0.6cm and 1.2cm);

\draw (8.4,6) arc(180:360:0.6cm and 1.2cm);
\draw (9.6,6.2) arc(0:180:0.6cm and 1.2cm);

\draw (8.4,-0.2) arc(180:360:0.6cm and 1.2cm);
\draw (9.6,0) arc(0:180:0.6cm and 1.2cm);


\draw (4.5,3) ellipse (0.6cm and 1.2cm);
\draw (7.5,3) ellipse (0.6cm and 1.2cm);
\draw (10.5,3) ellipse (0.6cm and 1.2cm);


\draw (2.4,6)--(3.6,6);
\draw (5.4,6)--(6.6,6);
\draw (8.4,6)--(9.6,6);

\draw (2.4,0)--(3.6,0);
\draw (5.4,0)--(6.6,0);
\draw (8.4,0)--(9.6,0);


\draw[fill=black] (4.5,3) circle (0.08cm);
 

 


 
\draw (9.9,3)--node {\midarrowa}(8.6,3); 
\draw (8.1,3)--node {\midarrowa}(9.4,3); 
\draw (9.9,3)--(8.1,3); 
\draw (8.1,3)--(9.9,3); 

\draw (6.9,3)--node {\midarrowa}(5.6,3); 
\draw (5.1,3)--node {\midarrowa}(6.4,3); 

\draw (6.9,3)--(5.1,3); 
\draw (5.1,3)--(6.9,3); 

\draw (3.5,5.4)--(5.5,6.9); 
\draw (6.5,5.4)--(8.5,6.9); 

\draw (5.5,-0.9)--(3.5,0.6); 

\draw (8.5,-0.9)--(6.5,0.6); 

\draw (0.75,5.4)--(2.5,5.4); 
\draw (2.5,0.6)--(0.75,0.6); 
\draw (0.95,5.4)--node {\midarrowa}(2.5,5.4); 
\draw (2.3,0.6)--node {\midarrowa}(0.75,0.6); 

\draw (5.5,0.6)--node {\midarrowa}(3.5,0.6); 
\draw (3.5,5.4)--node {\midarrowa}(5.5,5.4); 

\draw (8.5,0.6)--node {\midarrowa}(6.5,0.6); 
\draw (6.5,5.4)--node {\midarrowa}(8.5,5.4); 

\draw (6,4.8)--(6,1.2); 
\draw (9,4.8)--(9,1.2); 



\draw (4.5,3)--(3.4,0.91);  
\draw (3.4,5.13)--(4.5,3);  

\draw (4,2.04)--node{\midarrowa}(3.4,0.91);  
\draw (3.4,5.13)--node{\midarrowa}(3.95,4.0649);  

\draw (7.3,1.86)--node {\midarrowa}(6.4,0.91); 

\draw (10.3,1.86)--node {\midarrowa}(9.4,0.91); 

\draw (6.4,5.13)--node {\midarrowa}(7.3,4.14); 

\draw (9.4,5.13)--node {\midarrowa}(10.3,4.14); 

\draw (5.6,5.13)--node {\midarrowa}(4.7,4.14); 

\draw (8.6,5.13)--node {\midarrowa}(7.7,4.14); 

\draw (4.7,1.86)--node {\midarrowa}(5.6,0.91); 

\draw (7.7,1.86)--node {\midarrowa}(8.6,0.91); 


\draw (6,4.62)--(6,1.2); 
\draw (6,4.62)--node {\midarrowa}(6,3); 

\draw (9,4.62)--(9,1.2); 
\draw (9,4.62)--node {\midarrowa}(9,3); 


\node at (0.4,5.4) {$p$};
\node at (3,5.4) {$S_{1,2}^c$};
\node at (6,5.4) {$S_{2,3}^c$};
\node at (9,5.4) {$S_{3,4}^c$};

\node at (0.4,0.6) {$q$};
\node at (3,0.6) {$S_{2,1}^c$};
\node at (6,0.6) {$S_{3,2}^c$};
\node at (9,0.6) {$S_{4,3}^c$};

\node at (3.55,2.2) {$S_{2,2}$};
\node at (6.6,2.2) {$S_{3,3}$};
\node at (9.6,2.2) {$S_{4,4}$};

\node at (2,-1.3) {$S_{2,1}$};
\node at (5,-1.3) {$S_{3,2}$};
\node at (8,-1.3) {$S_{4,3}$};

\node at (2,7.3) {$S_{1,2}$};
\node at (5,7.3) {$S_{2,3}$};
\node at (8,7.3) {$S_{3,4}$};



\end{scope}
\end{tikzpicture}

    \caption{A coarse representation of $\ora{H_1}$ which shows the orientation of edges between various subsets of $V(G)$. A single arrow from one part to another indicates that all the edges between these parts are oriented from the former to latter. A double arrow between two parts indicates that the edges between the two parts are oriented in either direction or unoriented. An unoriented edge between two parts indicate that no edge between these two parts are oriented.}
    \label{fig1}
\end{figure}
\subsection{Oriented Diameter and $\boldsymbol{2}$-Step Domination Property of $\boldsymbol{\ora{H_1}}$}
Let $\orient{H_1}$ be the oriented subgraph of $G$ returned by the algorithm {\sc OrientedCore}. Since the smallest cycle containing $pq$ is of length greater than or equal to $5$, by Corollary~\ref{ddcorollary}, we can see that the diameter of $\ora{H_1}$ is at most $17$. Moreover, from equations \ref{eqnFromp} and \ref{eqnToq} of Lemma~\ref{lemmaDistancesInCore}, we get the upper bounds on the distances of $\orient{H_1}$ in Table~\ref{tab1}. Hence, the following corollary.  
\begin{corollary}\label{coroh1od}
$d(\ora{H_1}) \leq 17$. Moreover $\forall w \in V(\ora{H_1})$, $d_{\orient{H_1}}(p,w)$ and  $d_{\orient{H_1}}(w,q)$ obey the bounds in  Table~\ref{tab1}.   
\end{corollary}
\begin{remark}\label{rm1}
If $k > 5$ $(h > 2)$, then $S_{2,2}$ is empty. Moreover if $S_{2,2}$ is non-empty, then all the vertices in $S_{2,2}$ are captured into $\ora{H_1}$. 
\end{remark}
\begin{table}
  \begin{center}
      \begin{tabular}{| c || c | c | c | c | c | c | c | c | c |}
\hline

for $w$ in       & $S_{12}^c$  &  $S_{23}^c$  & $S_{34}^c$  &  $S_{22}$  &  $S_{33}^c$ & $S_{44}^c$ & $S_{21}^c$ & $S_{32}^c$ & $S_{43}^c$ \\\hline
$d_{\ora{H_1}}(p,w) \leq $ & 1  &  2  &  3 &   2  &  5  &  6 &   9 &  8 &  7   \\\hline
$d_{\ora{H_1}}(w,q) \leq $ & 7 & 6 & 5 & 2 & 5 & 4 & 1 & 2 & 3  \\
\hline
   
\end{tabular}
    \caption{Upper bounds on the distances of $\orient{H_1}$}
    \label{tab1}
  \end{center}
\end{table}
Furthermore, applying Lemma~\ref{lemmaDominationByCore} on $\ora{H_1}$ shows that $\ora{H_1}$ is a $2$-step dominating subgraph of $G$. Let $G_0$ be a graph obtained by contracting the vertices in $V(\ora{H_1})$ into a single vertex $v_H$. We can see that $G_0$ has radius at most $2$. Thus by Theorem~\ref{tt1}, $G_0$ has a strong orientation $\ora{G_0}$ with radius at most $6$. Since $d \leq 2r$, we have $d(\ora{G_0}) \leq 12$. Since $\ora{G_0}$ and $\ora{H_1}$ do not have any common edges we can see that $G$ has an orientation with diameter at most $29$ by combining the orientations in $\ora{H_1}$ and $\ora{G_0}$. But we further improve this bound to $21$ by constructing a $1$-step dominating oriented subgraph $\ora{H_2} $ of $G$. We propose the following asymmetric variant of a technique by  Chv\'atal and Thomassen \cite{chvatal1978distances} for the construction and analysis of $\ora{H_2}$. 

\subsection{Asymmetric Chv\'atal-Thomassen Lemma}
For any subset $A$ of $V(G)$, let $N(A)$ denote the set of all 
vertices with an edge incident on some vertex in $A$.
Let $H$ be a subgraph of $G$. An \emph{ear} of $H$ in $G$ is a sequence of edges $uv_1, v_1v_2, \ldots, v_{k-1}v_k, v_kv$ such that $u,v \in V(H)$, $k \geq 1$ and none of the vertices $v_1, \ldots , v_k$  and none of the edges in this sequence are in $H$. In particular we allow $u = v$.
 
\begin{lemma}[Asymmetric Chv\'atal-Thomassen Lemma]
\label{lemmaACT}
	Let $G$ be an undirected graph and let $A \subseteq B \subseteq V(G)$ such
	that 
	\begin{enumerate}[{\em(i)}]
		\item $B$ is a $k$-step dominating set in $G$,
		\item $G/B$ is $2$-edge connected, and 
		\item $N(A) \cup B$ is a $(k-1)$-step
		dominating set of $G$.
	\end{enumerate}
	Then there exists an oriented subgraph $\orient{H}$ of $G \setminus G[B]$ 
	such that
	\begin{enumerate}[{\em(i)}]
		\item $N(A) \setminus B \subseteq V(\orient{H})$ and hence $V(\orient{H}) \cup B$ is a $(k-1)$-step dominating set of $G$, and
		\item $\forall v \in V(\orient{H})$, we have $d_{\orient{H}}(A, v)
		 	\leq 2k$ and either $d_{\orient{H}}(v, A) \leq 2k$ or
		 	$d_{\orient{H}}(v, B \setminus A) \leq 2k-1$.
	\end{enumerate}

\end{lemma}

\begin{proof}
	We construct a sequence $\orient{H}_0, \orient{H}_1, \ldots$ of oriented
	subgraphs of $G \setminus G[B]$ as follows. We start with $\orient{H}_0 =
	\emptyset$ and add an oriented $A\--B$ ear $\orient{Q}_i$ in each
	step. Let $i \geq 0$. If $N(A) \setminus B \subseteq V(\orient{H_i})$, then we stop the
	construction and set $\orient{H} = \orient{H}_i$. Since $N(A) \cup B$ is a
	$(k-1)$-step dominating set of $G$,  the first conclusion of the lemma is
	satisfied when the construction ends with $N(A) \setminus B \subseteq V(\orient{H})$.
	If $N(A) \setminus B \not\subseteq V(\orient{H_i})$, then let $v \in (N(A) \setminus B) \setminus V(\orient{H}_i)$ and let $u$ be a
	neighbour of $v$ in $A$. Since $G/B$ is $2$-edge connected, there exists a
	path in $G' = (G/B) \setminus \{uv\}$ from $v$ to $B$. Let $P_i$ be a
	shortest $v\--B$ path in $G'$ with the additional property that once
	$P_i$ hits a vertex in an oriented ear $\orient{Q}_j$ that was added in a
	previous step, $P_i$ continues further to $B$ along the shorter arm of
	$Q_j$.  It can be verified that $P_i$ is still a shortest $v\--B$
	path in $G'$. The ear $Q_i$ is the union of the edge $uv$ and the path
	$P_i$. If $P_i$ hits $B$ without hitting any previous ear, then we orient
	$Q_i$ as a directed path $\orient{Q_i}$ from $u$ to $B$. If $Q_i \cap Q_j
	\neq \emptyset$, then we orient $Q_i$ as a directed path $\orient{Q}_i$ by
	extending the orientation of $Q_i \cap Q_j$.  Notice that, in both these
	cases, the source vertex of $\orient{Q}_i$ is in $A$.  We add
	$\orient{Q}_i$ to $\orient{H}_i$ to obtain $\orient{H}_{i+1}$.  

	Let $Q_i = (v_0, \ldots, v_q)$ with $v_0 \in A$ and $v_q \in B$ be the ear
	added in the $i$-th stage above. Since $(v_1, \ldots, v_q)$ is a shortest
	$v_1\--B$ path in $G' = (G/B) \setminus \{v_0v_1\}$ and since $B$ is
	a $k$-step dominating set, $q \leq 2k + 1$. Moreover, if $v_q \in B
	\setminus A$, then $q \leq 2k$ since $N(A) \cup B$ is a $(k-1)$-step
	dominating set.  These bounds on the length of $Q_i$ along with the
	observation that the source vertex of $\orient{Q}_i$ is in $A$, verifies
	the second conclusion of the lemma.
\end{proof}

\begin{remark}\label{rm2}
If we flip the orientation of $\ora{H}$ we get the bounds $d_{\orient{H}}(v, A)
		 	\leq 2k$ and either $d_{\orient{H}}(A, v) \leq 2k$ or
		 	$d_{\orient{H}}(B \setminus A, v) \leq 2k-1$, $\forall v \in V(\orient{H})$ in place of Conclusion (ii) of Lemma~\ref{lemmaACT}.		 	
		 	
\end{remark}		 	

Setting $A = B$ in Lemma~\ref{lemmaACT} gives the key idea which is recursively employed by Chv{\'a}tal and Thomassen to prove Theorem~\ref{tt1} \cite{chvatal1978distances}. Notice from the above proof that, in this case
	$B \subseteq V(\orient{H})$. We can summarize their idea as follows. 

\begin{lemma}[Chv\'atal-Thomassen Lemma]
\label{lemmaCT}
	Let $G$ be an undirected graph and let $B \subseteq V(G)$ such that 
	\begin{enumerate}[{\em(i)}]
		\item $B$ is a $k$-step dominating set in $G$, and
		\item $G/B$ is $2$-edge connected.
	\end{enumerate}
	Then there exists an oriented subgraph $\orient{H}$ of $G \setminus G[B]$ 
	such that
	\begin{enumerate}[{\em(i)}]
		\item $V(\orient{H})$ is a $(k-1)$-step dominating set of $G$, and
		\item $\forall v \in V(\orient{H})$, we have $d_{\orient{H}}(B, v)
		\leq 2k$ and $d_{\orient{H}}(v, B) \leq 2k$.
	\end{enumerate}

\end{lemma}

	Let $G$ be any $2$-edge connected graph with radius $r$. Chv\'atal and
	Thomassen showed that $\ora{d}(G) \leq 2r + 2(r-1) + \cdots + 2 = r(r+1)$ by
	$r$ applications of Lemma~\ref{lemmaCT}; starting with $B = \{v\}$, where
	$v$ is any central vertex of $G$ and $B$ in each subsequent application
	being the vertex-set of the oriented subgraph $\orient{H}$ returned by
	the current application.  

\subsection{A $\boldsymbol{1}$-Step Dominating Oriented Subgraph $\boldsymbol{\vec{H_2}}$ of $\boldsymbol{G}$}
Let $\ora{H_1}$ be the oriented subgraph of $G$ returned by the algorithm {\sc OrientedCore}. We will add further oriented ears to $\ora{H_1}$ to obtain a $1$-step dominating oriented subgraph $\ora{H_2}$ of  $G$. We have already seen that $\ora{H_1}$ is a $2$-step dominating oriented subgraph of $G$. By Lemma~\ref{lemmaDominationByCore}, we also have $d_i + d_j \leq 3$ for any distinct $i,j \in \{1,0,-1\}$.  

Now consider the first case where we have vertices in Level $0$ which are at a distance $2$ from $S_{2,2}$. Notice that in this case, $d_0 = 2$ and hence $d_1, d_{-1} \leq 1$. Let $B = L_{0}^c$, $A = S_{2,2}$ and $G_0 = G[L_0]$. By Remark~\ref{rm1}, $A \subseteq B$. Notice that $B = L_{0}^c$ is a cut-set separating $L_0^u$ from the rest of $G$ and hence the graph $G_0 / B $ is $2$-edge connected. Since $S_{3,3}^u \subseteq N(S_{2,2})$, we can see that $N(A) \cup B = N(S_{2,2}) \cup L_0^c $ is a $1$-step dominating subgraph of $G_0$. Therefore we can apply Lemma~\ref{lemmaACT} on $G_0$. Every edge of the oriented subgraph of $G_{0} \backslash G_{0}[B]$ obtained by applying Lemma~\ref{lemmaACT} is reversed to obtain the subgraph $\ora{H_2^0}$. Now consider the vertices captured into $\ora{H_2^0}$. From Lemma~\ref{lemmaACT} and Remark~\ref{rm2}, we get the following bounds $d_{\ora{H_2^0}}(v, A)
		 	\leq 4$ and either $d_{\ora{H_2^0}}(A, v) \leq 4$ or
		 	$d_{\ora{H_2^0}}(B \setminus A, v) \leq 3$, $\forall v \in V(\ora{H_2^0})$. Here $B \setminus A = S_{3,3}^c \cup S_{4,4}^c$ and from Table~\ref{tab1}, we have the bounds $d_{\ora{H_1}}(p, x) \leq 5$, $\forall x \in S_{3,3}^c$, $d_{\ora{H_1}}(p, y) \leq 6$, $\forall y \in S_{4,4}^c$ and $d_{\ora{H_1}}(p, z) = 2$, $\forall z \in S_{2,2}$. Hence $d_{\ora{H_1} \cup \ora{H_2^0}}(p, v) \leq 9$, $\forall v \in V(\ora{H_2^0})$. Since $d_{\ora{H_2^0}}(v, A) \leq 4$ and $d_{\ora{H_1}}(x, q) = 2$, $\forall x \in A$, we also have $d_{\ora{H_1} \cup \ora{H_2^0}}(v,q) \leq 6$, $\forall v \in V(\ora{H_2^0})$. Let $\ora{H_2} = \ora{H_1} \cup \ora{H_2^0}$. By the above discussion, in combination with the distances in Table~\ref{tab1} and Corollary~\ref{coroh1od}, we get the bounds in Table~\ref{tab2} for $d_{\ora{H_2}}(p,w)$ and  $d_{\ora{H_2}}(w,q)$ when $V(\ora{H_2^0}) \neq \phi$. Moreover, $d(\ora{H_2}) \leq 17$.

\begin{table}[h!]
  \begin{center}
      \begin{tabular}{| c || c | c | c | c | c | c | c | c | c |}
\hline

for $w$ in       & $S_{12}^c$  &  $S_{23}^c$  & $S_{34}^c$  &  $S_{22}$  &  $S_{33}^c$ & $S_{44}^c$ & $S_{21}^c$ & $S_{32}^c$ & $S_{43}^c$ \\\hline
$d_{\ora{H_2}}(p,w) \leq $ & 1  &  2  &  3 &   2  &  \textbf{9}  &  \textbf{9} &   9 &  8 &  7   \\\hline
$d_{\ora{H_2}}(w,q) \leq $ & 7 & 6 & 5 & 2 & \textbf{6} & \textbf{6} & 1 & 2 & 3  \\
\hline
   
\end{tabular}
    \caption{Upper bounds on the distances of $\ora{H_2}$ when $V(\ora{H_2^0}) \neq \phi$}
    \label{tab2}
  \end{center}
\end{table}

Now consider the second case where $d_1 = 2$ or $d_{-1} = 2$. Since $d_1 + d_{-1} \leq 3$, uncaptured vertices at a distance $2$ from $H_1$ can exist either in Level $1$ or in Level $-1$ but not both. By flipping the role of the vertices $p$ and $q$ in Algorithm {\sc OrientedCore} if necessary, without loss of generality, we can assume vertices which are at a distance 2 from $\ora{H_1}$ exists only in Level $-1$ and not in Level $1$. Let $G_{-1}=G[L_{-1}]$ and $B = L_{-1}^c$. Further let $r$ be any vertex in $S_{1,2}^c$ and $A = \{v \in B: d_{G}(r,v)=2\}$. Since $q\in A$, $A$ is never empty. Note that $A \subseteq B \subseteq V(G_{-1})$. Also $G_{-1} / B$ is $2$-edge connected since $B = L_{-1}^c$ is a cut-set which separates $L_{-1}^u$ from the rest of $G$. Now consider a vertex $z$ in Level $-1$ which is exactly at a distance $2$ from $B$. Since the graph $G$ is of diameter $4$, there exists a $4$-length path $P$ from $z$ to $r$. Since $B$ separates $L_{-1}^u$ from $r$, $P$ intersects $B$, say at a vertex $b$. Further, we have $d_G(b,r) = 2$ and thus $b \in A$. Hence $z$ has a $2$-length path to a vertex $b \in A$. Thus $N(A) \cup B$ is a $1$-step dominating subgraph of $G_{-1}$. Hence we can apply Lemma~\ref{lemmaACT} on $G_{-1}$ to obtain $\ora{H_2^{-1}}$, an oriented subgraph of $G_{-1} \backslash G_{-1}[B]$. Now consider the vertices captured into $\ora{H_2^{-1}}$. From Lemma~\ref{lemmaACT}, we get the following bounds $\forall v \in V(\ora{H_2^{-1}})$, $d_{\ora{H_2^{-1}}}(A, v)
		 	\leq 4$ and $d_{\ora{H_2^{-1}}}(v, B) \leq 4$. Since $d_{\ora{H_1}}(x,q) \leq 3$, $\forall x \in B$, we have $d_{\ora{H_1} \cup \ora{H_2^{-1}}}(v,q) \leq 7$, $\forall v \in V(\ora{H_2^{-1}})$. Vertices in $A$ can be from $S_{2,1}^c$, $S_{3,2}^c$ or $\{q\}$. By the definition of $A$ there is an undirected path in $G$ of length $3$ from $p$ to any vertex $v_a$ in $(A \setminus \{q\})$, going through $r$. It can be verified that this undirected path is oriented from $p$ to $v_a$ by Algorithm {\sc OrientedCore}. Hence $d_{\ora{H_1}}(p,v_a) \leq 3$, $\forall v_a \in (A \setminus \{q\})$ and hence $\forall v \in V(\ora{H_2^{-1}})$ with $d_{\ora{H_2^{-1}}}(A \setminus \{q\}, v) \leq 4$, $d_{\ora{H_1} \cup \ora{H_2^{-1}}}(p,v) \leq 7$. But if a vertex $v \in V(\ora{H_2^{-1}})$ has $d_{\ora{H_2^{-1}}}(A \setminus \{q\}, v) > 4$, then $d_{\ora{H_2^{-1}}}(q,v) \leq 4$. In this case, since $d_{\ora{H_1}}(p,q) \leq 8$, we get $d_{\ora{H_1} \cup \ora{H_2^{-1}}}(p,v) \leq 12$. Notice that this is the only situation where $d_{\ora{H_1} \cup \ora{H_2^{-1}}}(p,v) > 9$ and in this particular case $d_{\ora{H_2^{-1}}}(q,v) \leq 4$. 
		 	 
Now consider two vertices $x, y \in V(\ora{H_1} \cup \ora{H_2^{-1}})$. We can see that $d_{\ora{H_1} \cup \ora{H_2^{-1}}}(x,y)$ $ \leq d_{\ora{H_1} \cup \ora{H_2^{-1}}}(x,q) + d_{\ora{H_1} \cup \ora{H_2^{-1}}}(q,y)$. We have already proved that $d_{\ora{H_1} \cup \ora{H_2^{-1}}}(x,q) \leq 7$. Now let us consider the $q-y$ path. If $y \in V(\ora{H_1})$, from Table~\ref{tab1}, we can see that $d_{\ora{H_1}}(p,y) \leq 9$ and therefore $d_{\ora{H_1} \cup \ora{H_2^{-1}}}(x,y) \leq 17$. Now suppose if $y \in (V(\ora{H_2^{-1}}) \setminus V(\ora{H_1}))$. In this case we have already shown that $d_{\ora{H_2^{-1}}}(p,y) \leq 9$ or $d_{\ora{H_2^{-1}}}(q,y) \leq 4$. So, we either have a directed path of length $10$ from $q$ to $y$ through $p$ or a directed path of length $4$ to $y$ directly from $q$. Hence, $d_{\ora{H_1} \cup \ora{H_2^{-1}}}(x,y) \leq 17$. Let $\ora{H_2} = \ora{H_1} \cup \ora{H_2^{-1}}$. By the above discussion, we get the bounds in Table~\ref{tab3} for $d_{\ora{H_2}}(p,w)$ and  $d_{\ora{H_2}}(w,q)$ when $V(\ora{H_2^{-1}}) \neq \phi$. Moreover, $d(\ora{H_2}) \leq 17$.

\begin{table}[h!]
  \begin{center}
      \begin{tabular}{| c || c | c | c | c | c | c | c | c | c |}
\hline

for $w$ in       & $S_{12}^c$  &  $S_{23}^c$  & $S_{34}^c$  &  $S_{22}$  &  $S_{33}^c$ & $S_{44}^c$ & $S_{21}^c$ & $S_{32}^c$ & $S_{43}^c$ \\\hline
$d_{\ora{H_2}}(p,w) \leq $ & 1  &  2  &  3 &   2  &  5  &  6 &   \textbf{12} &  \textbf{12} &  \textbf{12}   \\\hline
$d_{\orient{H_2}}(w,q) \leq $ & 7 & 6 & 5 & 2 & 5 & 4 & \textbf{7} & \textbf{7} & \textbf{7}  \\
\hline
   
\end{tabular}
    \caption{Upper bounds on the distances of $\ora{H_2}$ when $V(\orient{H_2^{-1}}) \neq \phi$}
    \label{tab3}
  \end{center}
\end{table}   

In both the cases we get an oriented subgraph $\ora{H_2}$ of $G$ with $d(\ora{H_2}) \leq 17$. Moreover, it is clear from Conclusion (i) of Lemma~\ref{lemmaACT} that $\ora{H_2}$ is a $1$-step dominating subgraph of $G$. Hence the following Lemma.
\begin{lemma}\label{lemmaH2}
For every $2$-edge connected graph $G$ with diameter $4$ and $\eta(G) \geq 5$, there exists a $1$-step dominating oriented subgraph $\ora{H_2}$ of $G$ with $d(\ora{H_2}) \leq 17$.   
\end{lemma}    

\subsection{The Upper Bound}
Now the main theorem of the section follows.               

\begin{theorem}\label{betterfor4}
$f(4) \leq 21$.
\end{theorem}

\begin{proof}
By Lemma~\ref{lemmaH2}, we get a $1$-step dominating oriented subgraph $\ora{H_2}$ of $G$ with $d(\ora{H_2}) \leq 17$. Let $G_0$ be a graph obtained by contracting the vertices in $V(\ora{H_2})$ into a single vertex $v_H$. We can see that $G_0$ has radius at most $1$. Thus by Theorem~\ref{tt1}, $G_0$ has a strong orientation $\ora{G_0}$ with radius at most $2$. Since $d \leq 2r$, we have $d(\ora{G_0}) \leq 4$. Notice that $\ora{G_0}$ and $\ora{H_2}$ do not have any common edges. Now we can see that $G$ has an orientation with diameter at most $21$ by combining the orientations in $\ora{H_2}$ and $\ora{G_0}$.
\end{proof}


\end{document}